\documentclass[11pt]{article}
\usepackage[tmargin=2.0cm,bmargin=2.0cm,lmargin=2.3cm,rmargin=2.3cm]{geometry}
\usepackage{amsmath,amssymb,amsfonts,amsthm,mathrsfs}
\usepackage{lmodern}
\usepackage[shortlabels]{enumitem}
\setlist[enumerate]{leftmargin=*,labelindent=0pt,itemsep=0.3ex,topsep=0.4ex,parsep=0pt,partopsep=0pt}
\usepackage{titlesec}
\titleformat{\section}
  {\normalfont\fontsize{13.5pt}{16pt}\selectfont\bfseries}
  {\thesection}{1em}{}
\usepackage[hidelinks]{hyperref}
\usepackage[capitalise]{cleveref}
\hypersetup{pdftitle={Saturation numbers for joins with graphs containing isolated vertices},
pdfauthor={Xinying Hua and Yuejian Peng},pdfkeywords={Saturation number, Saturated graph, Join of graphs, Matching}}
\newtheorem{theorem}{Theorem}[section]
\newtheorem{lemma}{Lemma}[section]
\newtheorem{claim}{Claim}[section]
\newtheorem{proposition}{Proposition}[section]
\newtheorem{observation}{Observation}[section]
\newtheorem{problem}{Problem}
\newcommand{\sat}{\mathrm{sat}}
\newcommand{\SAT}{\mathrm{SAT}}
\newcommand{\diam}{\mathrm{diam}}
\newcommand{\odd}{\mathrm{odd}}

\title{Saturation numbers for joins with graphs containing isolated vertices}
\author{Xinying Hua \ \ Yuejian Peng\thanks{E-mail addresses: xyhuamath@163.com (X. Hua), ypeng1@hnu.edu.cn (Y. Peng, corresponding author).}\\
{\footnotesize School of Mathematics, Hunan University, Changsha, Hunan, 410082, P.R. China}\\
}
\date{}

\begin{document}
\maketitle
\begin{abstract}
A graph $G$ is $H$-saturated if $G$ contains no $H$-copy as a subgraph, but adding any edge between two non-adjacent vertices in $G$ creates a copy of $H$. The saturation number $\sat(n,H)$ is the minimum number of edges in an $n$-vertex $H$-saturated graph.

Cameron and Puleo [Discrete Math. 345 (2022), 112867] showed that $\sat(n,K_1\vee F)\le n-1+\sat(n-1,F)$ for all $n>|V(F)|$. It is natural to ask when the equality holds. Most known results on this bound consider non-empty graphs $F$ without isolated vertices. We study  families in which $F$ contains isolated vertices.

 For $F=K_{p-1}\cup qK_1$, where $p\ge4$ and $q\ge p-1$, we prove that $\sat(n,K_1\vee F)=(p-2)n-\binom{p-1}{2}$ for $n\ge\max\{p+q,\binom p2\}$ and determine the unique extremal graph. We also give the exact result for $q=1$, where the extremal graph has a different structure.

For $F=pK_2\cup qK_1$, in the natural range $n\ge2p+q+1$, we prove that $\sat(n,K_1\vee F)=n+3p-4$ holds exactly when $n\ge3p-2$, and the order condition is sharp.  Furthermore, we characterize all extremal graphs.
\end{abstract}
\begin{flushleft}
\end{flushleft}

\section{Introduction}
All graphs  in this paper are simple graphs. Let $G=(V(G),E(G))$ be a graph.
For graphs $F$ and $G$, $G$ is said to be \emph{$F$-saturated} if $G$ contains no copy of $F$ as a subgraph, but for any $e\in E(\overline{G})$, $G + e$ (obtained by adding the edge $e$ to $G$) produces a copy of $F$. The \emph{saturation number} of $F$, denoted by $\sat(n,\,F)$, is the minimum number of edges in an $F$-saturated graph of order $n$. Further,
\[
\SAT(n,\,F):= \{G : |V(G)|=n,  |E(G)|=\sat(n, F), \mbox{ and } G \mbox{ is } F\mbox{-saturated}\}.
\]
Each graph in $\SAT(n,\,F)$ is called an \emph{extremal graph} for $F$.
A vertex of degree one is called a \emph{leaf}, and the unique edge incident to a leaf is called a \emph{pendant edge}.
A vertex $v\in V(G)$ is called a \emph{full-degree vertex} if $v$ is adjacent to each other vertex of $G$.
We will denote the path, cycle, star, complete graph and empty graph on $n$ vertices by $P_n$, $C_n$, $S_n$, $K_n$ and $\overline{K_n}$, respectively. We let $K_s^{+t}$ denote the graph obtained by attaching $t$ leaves to the same vertex of $K_s$.
For any two vertex-disjoint graphs $G_1$ and $G_2$, let $G_1\vee G_2$ denote the \emph{join graph} obtained from $G_1 \cup G_2$ by adding all edges between $V(G_1)$ and $V(G_2)$.
For simplicity, we let
$$J_{p,q}=K_1\vee(pK_2\cup qK_1),$$ where $p,q\ge1$. We use $K_0$ for the graph with no vertices.

The notion of the saturation number of a graph was first introduced by Erd\H{o}s, Hajnal and Moon~\cite{erdos1964}. They showed that $\sat(n, K_{k+1})=(k-1)n-\binom{k}{2}$ with $K_{k-1}\vee \overline{K_{n-k+1}}$ as the unique extremal graph. Later, K\'aszonyi and Tuza \cite{Kaszonyi} provided a general upper bound on the saturation number for any graph, and determined $\sat(n,P_k)$, $\sat(n, S_k)$ and $\sat(n,tK_2)$. Following these foundational results, many studies have investigated saturation numbers for various graph families, such as cycles \cite{Chen2009,Chen2011,Lan2025,Ma2021,Ollmann1972,Tuza1989}, trees \cite{Faudree2009}, forests \cite{Cao2023,Chen2015,Fan2015,Faudree2009,He2023,Kaszonyi,Lv2023}, complete multipartite graphs \cite{Chen2014,He2021,Huang2024}, and cliques \cite{Chen2024,Faudree2009b,Zhu2025}.
Despite these advances, $\sat(n, F)$ and $\SAT(n, F)$ have been determined for only a few graph families. Comprehensive summaries of known results can be found in \cite{Currie2021}.

Most work on join graphs $K_s \vee F$ concentrates on deriving $\sat(n,K_1 \vee F)$ based on the saturation number of $F$.
Cameron and Puleo \cite{Ca} showed that $\sat(n,K_1 \vee F)\le n-1+\sat(n-1, F)$ for all $n > |V(F)|$.
An interesting problem is to find graphs $F$ such that the above equality holds. Recently, Hu, Luo and Peng \cite{HU} proved the following theorem.
	
\begin{theorem}\label{T12} \text{\bf \cite{HU}} Let $s ,n$ be positive integers and $F$ be a graph without isolated vertices. Then for $n \geq 3s^2 -s +2\,\sat(n-s, F)+1$, we have
\begin{equation*}
\begin{aligned}
\sat(n, K_s \vee F) = \binom{s}{2} + s(n-s) + \sat(n-s, F) \, .
\end{aligned}
\end{equation*}
\end{theorem}
Qiu, He, Lu and Xu~\cite{Qiu} commented that the condition on $n$ in Theorem~\ref{T12} implies that $F$ must contain an isolated edge.
In  \cite{HU1}, Hu, Ji and Cui showed that $\sat(n,K_1 \vee P_t)=n-1+\sat(n-1,P_t)$ holds for $t \geq 5$ and sufficiently large $n$.
In \cite{Song}, Song, Hu, Ji and Cui proved that $\sat(n,K_1 \vee C_4)=\lfloor\frac{5n-10}{2}\rfloor$ holds for all $n \geq 6$, and gave a complete characterization of the extremal graphs.
In~\cite{Qiu}, Qiu, He, Lu and Xu further proved that $\sat(n,K_1 \vee C_t)=n-1+\sat(n-1, C_t)$ holds for $t \geq 8$ and $n \geq 56t^3$.

Very recently, He, Lu, Qiu and Xu~\cite{He2026} obtained the following general result, which includes several previous results as special cases.

\begin{theorem}\label{thm:HeLuQiuXu}\text{\bf \cite{He2026}}
Let $s\ge1$ and $t\ge2$ be fixed integers. Let $F$ be a graph without isolated vertices such that each edge of $F$ belongs to a $K_t$ and $\sat(n,F)<(t-1)n-\binom{t}{2}-2$ for sufficiently large $n$. Then for sufficiently large $n$,
\[\sat(n,K_s\vee F)=s(n-s)+\sat(n-s,F)+\binom{s}{2}.\]
Moreover, each extremal graph for $K_s\vee F$ contains $s$ full-degree vertices.
\end{theorem}
However, this result also requires $F$ to have no isolated vertices.
In a different direction, Chen~\cite{GChen2009} considered the case where $F$ is an empty graph and determined the exact value of $\sat(n,K_s\vee pK_1)$ for sufficiently large $n$.
These results motivate us to consider the case where $F$ is a non-empty graph containing isolated vertices.
In this paper, we consider the saturation number of $K_1\vee F$ for such graphs $F$. We determine the exact saturation numbers and characterize all extremal graphs for the two classes of graphs considered below.

Our results also show that the isolated vertices in $F$ may affect the structure of the extremal graphs. For $F=K_{p-1}\cup qK_1$, if $q\ge p-1$, the extremal graph is connected and contains $p-2$
full-degree vertices, while for $q=1$, the unique extremal graph is a disjoint union of complete graphs.
For $F=pK_2\cup qK_1$, the extremal graph is connected and contains a full-degree vertex, except for the case $(p,q)=(1,1)$, where both connected and disconnected extremal graphs may occur.

We first consider $F=K_{p-1}\cup qK_1$. Recall that
\[
K_1\vee(K_{p-1}\cup qK_1)=K_p^{+q}.
\]
The following theorem determines the saturation number and the unique extremal graph for $q\ge p-1$.

\begin{theorem}\label{thm:Kp+}
Let $p\ge4$, $q\ge p-1$, and
\[
n\ge\max\left\{p+q,\binom p2\right\}.
\]
Then
\[
\sat(n,K_p^{+q})=(p-2)n-\binom{p-1}{2}.
\]
Furthermore,
\[
\SAT(n,K_p^{+q})=
\left\{K_{p-2}\vee\overline K_{n-p+2}\right\}.
\]
\end{theorem}

The extremal graph in Theorem~\ref{thm:Kp+} is connected and has $p-2$ full-degree vertices. In contrast, for $q=1$ and $n\ge p+1$, the unique extremal graph for $K_p^{+q}$ is a disjoint union of complete graphs; see Proposition~\ref{prop:one-leaf}.

We next consider $F=pK_2\cup qK_1$. The result holds for all positive integers $p$ and $q$.

\begin{theorem}\label{thm:matching-join}
Let $p,q\ge1$. If
\[
n\ge\max\{2p+q+1,3p-2\},
\]
then
\[
\sat(n,J_{p,q})=n+3p-4.
\]
If $(p,q)\ne(1,1)$, then
\[
\SAT(n,J_{p,q})=
\left\{K_1\vee\bigl((p-1)K_3\cup\overline K_{n-3p+2}\bigr)\right\}.
\]
If $(p,q)=(1,1)$, then
\[
\SAT(n,J_{1,1})=
\{aK_3\cup S_{n-3a}:a\ge0,\ n>3a,\ n-3a\ne3\},
\]
where $a$ is an integer. Moreover, if $2p+q+1\le n<3p-2$, then
\[
\sat(n,J_{p,q})>n+3p-4.
\]
\end{theorem}

When $q\ge p-3$, the condition $n\ge |V(J_{p,q})|=2p+q+1$ already gives $n\ge3p-2$. When $q\le p-4$, we show that $\sat(n,J_{p,q})\ne n+3p-4$ for $2p+q+1\le n\le3p-3$. Hence the condition $n\ge3p-2$ is sharp. We give the proof in Subsection~\ref{subsec:sharpness}.

\subsection*{Notation}
For a graph $G$, let $e(G)=|E(G)|$. For two disjoint subsets $A,B\subseteq V(G)$, let $e(A,B)$ denote the number of edges with one endpoint in $A$ and the other in $B$.
For $v\in V(G)$, let $N_G(v)$ be its set of neighbours and let $N_G[v]=N_G(v)\cup\{v\}$. Also, $d_G(v)=|N_G(v)|$, $\Delta(G)=\max_{v\in V(G)}d_G(v)$, and $\delta(G)=\min_{v\in V(G)}d_G(v)$.
For a vertex set $U$, write $N_U(v)=N_G(v)\cap U$ and $d_U(v)=|N_U(v)|$.
The graph $G-U$ is obtained by deleting the vertices of $U$, and $G[U]$ is the subgraph induced by $U$.
A matching is a set of pairwise disjoint edges, and $\nu(G)$ denotes the maximum size of a matching in $G$. We write $\odd(G)$ for the number of components of $G$ with an odd number of vertices.
For a connected graph $G$, let $d(u,v)$ be the length of a shortest $(u,v)$-path and let $\diam(G)=\max_{u,v\in V(G)}d(u,v)$. We use $[n]=\{1,2,\ldots,n\}$.

The rest of this paper is organized as follows. In Section~\ref{sec:cliques}, we prove Theorem~\ref{thm:Kp+} and consider the case $q=1$.
In Section~\ref{sec:matchings}, we prove Theorem~\ref{thm:matching-join} and show that the order condition is sharp. We conclude with several observations and open problems in Section~\ref{sec:conclusion}.

\section{Saturation number for $K_p^{+q}$ }\label{sec:cliques}
We first prove Theorem~\ref{thm:Kp+}. We then compare it with the case $q=1$, for which the extremal graph has a different structure.
\subsection{The case \texorpdfstring{$q\ge p-1$}{q at least p-1}}

\begin{proof}[\textbf{Proof of Theorem~\ref{thm:Kp+}}]
We first prove the upper bound. Let $G_0=K_{p-2}\vee\overline K_{n-p+2}$, and let $C$ be the set of all full-degree vertices of $G_0$. Then $|C|=p-2$.
Since $G_0$ contains no $K_p$-copy, $G_0$ is $K_p^{+q}$-free. For any $xy\in E(\overline G_0)$, $x$ and $y$ belong to the independent set. Then $G_0[C\cup\{x,y\}]+xy\cong K_p$.
Fix a vertex $z\in C$. Since $|V(G_0)\backslash(C\cup\{x,y\})|=n-p\ge q$ and $z$ is
full-degree, $G_0+xy$ contains a $K_p^{+q}$-copy. Thus, $G_0$ is $K_p^{+q}$-saturated, and
\begin{equation} \label{eq:Kpq-upper}
\sat(n,K_p^{+q})\le (p-2)n-\binom{p-1}{2}.
\end{equation}

Let $G$ be a $K_p^{+q}$-saturated graph of order $n$ such that
\begin{equation}\label{eq:Kpq-assumption}
e(G)\le (p-2)n-\binom{p-1}{2}.
\end{equation}
We show that $G\cong G_0$.

Since $n\ge p+q$ and $G$ is $K_p^{+q}$-free, $G$ is not complete. If there exists $x\in V(G)$ such that $G[N_G(x)]$ contains a $K_{p-1}$-copy, then $d_G(x)\le p+q-2$. Indeed, if $d_G(x)\ge p+q-1$, then
$G[N_G[x]]$ contains a $K_p^{+q}$-copy, a contradiction.
Let
\[
A=\{x\in V(G):d_G(x)=p+q-2
\text{ and }G[N_G(x)]\text{ contains a }K_{p-1}\text{-copy}\},
\]
and let $B=V(G)\backslash A$.

\begin{claim}\label{clm:Kpq-common}
For any $xy\in E(\overline G)$ with $x,y\in B$,
$G[N_G(x)\cap N_G(y)]$ contains a $K_{p-2}$-copy.
\end{claim}

\begin{proof}
Since $xy\in E(\overline G)$, $G+xy$ contains a $K_p^{+q}$-copy $M$ containing $xy$. Suppose that $xy$ is a pendant edge of $M$, and let $z\in\{x,y\}$ with $d_M(z)=p+q-1$.
Then $G[N_G(z)]$ contains a $K_{p-1}$-copy and $d_G(z)\ge p+q-2$. Since $G$ is $K_p^{+q}$-free, $d_G(z)\le p+q-2$. Thus, $z\in A$, a contradiction.
Therefore $xy$ is an edge of the $K_p$-copy in $M$. Then $G[N_G(x)\cap N_G(y)]$ contains a $K_{p-2}$-copy.
\end{proof}

Let $v$ be a vertex of $G$ with $d_G(v)=\delta(G)$, and let
$T=V(G)\backslash N_G[v]$. By \eqref{eq:Kpq-assumption},
$e(G)<(p-2)n$. So, $\delta(G)\le\frac{2e(G)}{n}<2p-4$,
and therefore $\delta(G)\le2p-5$. Since $q\ge p-1$ and each
vertex in $A$ has degree $p+q-2\ge2p-3$, we have $v\in B$.

\begin{claim}\label{clm:Kpq-outside}
$T\cap B\ne\emptyset$.
\end{claim}

\begin{proof}
Suppose that $T\cap B=\emptyset$. Then $T\subseteq A$. Thus,
$d_G(x)=p+q-2\ge2p-3$ for each $x\in T$.
Since each vertex in $N_G[v]$ has degree at least $\delta(G)$, $|N_G[v]|=\delta(G)+1$ and $|T|=n-\delta(G)-1$,
we have
\[
\begin{aligned}
2e(G)&=\sum_{x\in V(G)\backslash T}d_G(x)+\sum_{x\in T}d_G(x)\\
&\ge \delta(G)(\delta(G)+1)+(2p-3)(n-\delta(G)-1)\\
&=(2p-3)n+\delta(G)^2-(2p-4)\delta(G)-(2p-3).
\end{aligned}
\]
Moreover, $\delta(G)^2-(2p-4)\delta(G)=(\delta(G)-p+2)^2-(p-2)^2\ge -(p-2)^2.$
Since $n\ge p+q$, we have
$n-p+1\ge q+1>0$. Hence,
\[
\begin{aligned}
2e(G)
&\ge (2p-3)n-(p-2)^2-(2p-3)\\
&=(2p-3)n-(p-1)^2\\
&=(2p-4)n-(p-1)(p-2)+(n-p+1)\\
&>(2p-4)n-(p-1)(p-2),
\end{aligned}
\]
a contradiction to \eqref{eq:Kpq-assumption}.
\end{proof}
By Claim~\ref{clm:Kpq-outside}, we can choose
$x\in T\cap B$. Then $vx\in E(\overline G)$ and $v,x\in B$.
By Claim~\ref{clm:Kpq-common},
$G[N_G(v)\cap N_G(x)]$ contains a $K_{p-2}$-copy.
Therefore,
$\delta(G)\ge p-2$ and $e(G[N_G(v)])\ge\binom{p-2}{2}.$

For each $x\in T$, $d_G(x)=d_{N_G(v)}(x)+d_T(x)$. Then
\begin{equation}\label{eq:Kpq-count}
\begin{aligned}
e(G)
&=\delta(G)+e(G[N_G(v)])+e(N_G(v),T)+e(G[T])\\
&=\delta(G)+e(G[N_G(v)])+\frac{1}{2}\sum_{x\in T}\Big(2d_{N_G(v)}(x)+d_T(x)\Big)\\
&=\delta(G)+e(G[N_G(v)])+\frac{1}{2}\sum_{x\in T}\Big(d_G(x)+d_{N_G(v)}(x)\Big).
\end{aligned}
\end{equation}

Suppose first that $\delta(G)=p-2$. Since $e(G[N_G(v)])\ge\binom{p-2}{2}$, we have
$G[N_G(v)]\cong K_{p-2}$.
For each $x\in T\cap B$, Claim~\ref{clm:Kpq-common} implies that $G[N_G(v)\cap N_G(x)]$ contains a $K_{p-2}$-copy. Since $|N_G(v)|=p-2$, $x$ is adjacent to each vertex of $N_G(v)$.
Hence, $d_G(x)+d_{N_G(v)}(x)\ge2p-4.$
For each $x\in T\cap A$, we have $d_G(x)+d_{N_G(v)}(x)\ge d_G(x)=p+q-2\ge2p-3$.
Since $|T|=n-p+1$, by \eqref{eq:Kpq-count},
\[
\begin{aligned}
e(G)&\ge p-2+\binom{p-2}{2}+\frac{2p-3}{2}|T\cap A|+\frac{2p-4}{2}|T\cap B|\\
&=p-2+\binom{p-2}{2}+(p-2)|T|+\frac{|T\cap A|}{2}\\
&=(p-2)n-\binom{p-1}{2}+\frac{|T\cap A|}{2}.
\end{aligned}
\]
Combining this with \eqref{eq:Kpq-assumption}, we have
$T\cap A=\emptyset$ and equality holds throughout. Hence, for each
$x\in T$, we have $x\in B$ and $d_{N_G(v)}(x)=d_G(x)=p-2$.
Thus, each vertex of $T$ is adjacent to each vertex of $N_G(v)$
and has no neighbor in $T$. Hence $T\cup\{v\}$ is independent and
$G\cong G_0.$

Now assume that $p-1\le\delta(G)\le2p-5$. Let $\delta(G)=p-2+r$, where $1\le r\le p-3$.
For each $x\in T\cap A$, we have $d_G(x)+d_{N_G(v)}(x)\ge d_G(x)=p+q-2\ge2p-3.$
For each $x\in T\cap B$, by Claim~\ref{clm:Kpq-common}, $d_{N_G(v)}(x)\ge p-2$. Since
$d_G(x)\ge\delta(G)=p-2+r$, we have $d_G(x)+d_{N_G(v)}(x)\ge2p-4+r.$
Thus, by \eqref{eq:Kpq-count} and $e(G[N_G(v)])\ge\binom{p-2}{2}$,
\begin{equation}\label{eq:Kpqp-1count1}
\begin{aligned}
2e(G)&\ge 2(p-2+r)+(p-2)(p-3)+(2p-3)|T\cap A|+(2p-4+r)|T\cap B|\\
&=(p-1)(p-2)+2r+(2p-3)|T\cap A|+(2p-4+r)|T\cap B|.
\end{aligned}
\end{equation}
Since $n=p-1+r+|T\cap A|+|T\cap B|$, by \eqref{eq:Kpq-assumption},
\begin{equation}\label{eq:Kpqp-1count2}
\begin{aligned}
2e(G)&\le(2p-4)n-(p-1)(p-2)\\
&=(2p-4)\bigl(p-1+r+|T\cap A|+|T\cap B|\bigr)-(p-1)(p-2)\\
&=(p-1)(p-2)+(2p-4)\bigl(r+|T\cap A|+|T\cap B|\bigr).
\end{aligned}
\end{equation}
Combining \eqref{eq:Kpqp-1count1} and \eqref{eq:Kpqp-1count2}, we obtain
\begin{equation}\label{eq:clique-edge-count}
|T\cap A|+r|T\cap B|\le2r(p-3).
\end{equation}

Also, each vertex of $N_G[v]$ has degree at least $\delta(G)=p-2+r$. Each vertex of $T\cap A$ has degree $p+q-2\ge2p-3$, and each vertex of $T\cap B$ has degree at least $p-2+r$. Hence,
\[
\begin{aligned}
2e(G)
&\ge(p-2+r)(p-1+r)+(2p-3)|T\cap A|+(p-2+r)|T\cap B|\\
&=(p-1)(p-2)+(2p-3)r+r^2
 +(2p-3)|T\cap A|+(p-2+r)|T\cap B|.
\end{aligned}
\]
Together with \eqref{eq:Kpqp-1count2},
we have
\begin{equation}\label{eq:clique-degree-count}
|T\cap A|-(p-2-r)|T\cap B|
\le-r(r+1).
\end{equation}
Since $p-1-r>0$ and $r-1\ge0$, multiplying
\eqref{eq:clique-edge-count} by $p-1-r$ and
\eqref{eq:clique-degree-count} by $r-1$, and then adding them, gives
\[
\begin{aligned}
(p-2)|T|
&\le 2r(p-3)(p-1-r)-r(r+1)(r-1)\\
&=2r(p-3)(p-1-r)-r(r-1)^2-2r(r-1)\\
&=2r(p-2)(p-2-r)-r(r-1)^2.
\end{aligned}
\]
Since $p\ge4$, we have $p-2>0$. Thus,
\begin{equation}\label{eq:clique-order}
\begin{aligned}
n
&=p-1+r+|T|\\
&\le p-1+r+2r(p-2-r)
-\frac{r(r-1)^2}{p-2}\\
&=p-1+(2p-3)r-2r^2
-\frac{r(r-1)^2}{p-2}\\
&=\binom p2
-\frac{(p-2-2r)(p-1-2r)}{2}
-\frac{r(r-1)^2}{p-2}.
\end{aligned}
\end{equation}

Suppose that $p\ge5$. Since $p-2-2r$ and $p-1-2r$ are two consecutive integers, $(p-2-2r)(p-1-2r)\ge0.$ If $r=1$, then $(p-2-2r)(p-1-2r)=(p-4)(p-3)>0.$
If $r\ge2$, then $\frac{r(r-1)^2}{p-2}>0$.
Thus, in each case, \eqref{eq:clique-order} gives
$n<\binom p2$, contradicting $n\ge\binom p2$.
If $p=4$, since $1\le r\le p-3$, we have $r=1$. By \eqref{eq:clique-order}, we have $n\le6$.
However, $n\ge p+q$ and $q\ge p-1$, and hence $n\ge7$, again a contradiction.

Thus, $\delta(G)=p-2$. By the argument above, $G\cong G_0$. Together with \eqref{eq:Kpq-upper}, we obtain $\sat(n,K_p^{+q})=e(G_0)=(p-2)n-\binom{p-1}{2}$
and $\SAT(n,K_p^{+q})=\left\{K_{p-2}\vee\overline K_{n-p+2}\right\}.$
This completes the proof.
\end{proof}

\subsection{The case \texorpdfstring{$q=1$}{q=1}}
For $K_p^{+1}$, the unique extremal graph is a disjoint union of complete graphs.

\begin{proposition}\label{prop:one-leaf}
Let $p\ge4$, $n\ge p+1$, and let $0\le k\le p-1$ satisfy $n\equiv k\pmod p$. Then
\[
\sat(n,K_p^{+1})=\frac{(p-1)n}{2}+\frac{k(k-p)}{2},
\]
and
\[
\SAT(n,K_p^{+1})=
\left\{\left\lfloor\frac np\right\rfloor K_p\cup K_k\right\}.
\]
\end{proposition}

\begin{proof}
Recall that $n\equiv k\pmod{p}$. Let
\[h(n,p)=\frac{(p-1)n}{2}+\frac{k(k-p)}{2}.\]
Then $\lfloor\frac{n}{p}\rfloor K_p\cup K_k$ is a $K_p^{+1}$-saturated graph with $h(n,p)$ edges. Thus, $\sat(n, K_p^{+1})\le h(n,p)$.

Let $G$ be any $K_p^{+1}$-saturated graph of order $n$. Now we show that $e(G)\ge h(n,p)$ and equality holds if and only if $G\cong\lfloor\frac{n}{p}\rfloor K_p\cup K_k$.

We consider the following two cases.

\noindent\textbf{Case 1.} $G$ is connected.

If $G$ contains $K_p$ as a subgraph, since $n\ge p+1$ and $G$ is connected, $G$ contains a $K_p^{+1}$-copy, a contradiction. So, $G$ is $K_p$-free. Since $n\ge p+1$, $G$ is not complete. Let $u$ be a vertex in $G$ such that $d_G(u)=\delta(G)$.
\begin{claim}\label{clm:KP+}
For any $v\notin N_G[u]$, $G[N_G(u)\cap N_G(v)]$ contains a $K_{p-2}$-copy.
\end{claim}
\begin{proof}
Since $uv\in E(\overline{G})$, $G+uv$ contains a $K_p^{+1}$-copy $M$ such that $uv\in E(M)$.
Further, $uv$ is an edge in the $K_p$-copy of $M$, since $G$ contains no $K_p$-copy. Then $G[N_G(u)\cap N_G(v)]$ contains a $K_{p-2}$-copy.
\end{proof}
If $\delta(G)\le p-2$, then by Claim~\ref{clm:KP+}, $d_G(u)=p-2$, and $G[N_G(u)\cap N_G(v)]\cong K_{p-2}$ for any $v\in V(G)\backslash N_G[u]$. Thus, $G$ contains $K_{p-2}\vee \overline{K_{n-p+2}}$ as a subgraph. As $G$ is $K_p$-free, $G\cong K_{p-2}\vee \overline{K_{n-p+2}}$ and
\[e(G)=\binom{p-2}{2}+(n-p+2)(p-2)=(p-2)n+\frac{-p^2+3p-2}{2}=:h_1(n,p).\]
A direct calculation gives $h_1(n,p)> h(n,p)$ when $n>p+ \frac{k(k-p)+2}{p-3}$. Since $k< p$, if $k\ge 1$, then $ n\ge p+1> p+ \frac{k(k-p)+2}{p-3}$. Otherwise, assume that $k=0$. Since $n\ge p+1$, $n\ge 2p>p+\frac{2}{p-3}$. Therefore, $h_1(n,p)> h(n,p)$. Consequently, $e(G)>h(n,p)$.

If $\delta(G)\ge p-1$, then $e(G)\ge\frac{(p-1)n}{2}\ge h(n,p)$. Suppose that equality holds. Then $G$ is a $(p-1)$-regular graph with $n\equiv 0\pmod{p}$. Since $n\ge p+1$, we have $n\ge 2p$ and $|V(G)\backslash N_G[u]|=n-p\ge p>2$.
Let $N_G(u)=\{u_1, u_2,\ldots, u_{p-1}\}$ and take $v_1, v_2\in V(G)\backslash N_G[u]$. By Claim~\ref{clm:KP+}, $G[N_G(u)\cap N_G(v_1)]$ contains a $K_{p-2}$-copy $M_1$. Assume that $V(M_1)=\{u_1, u_2,\ldots, u_{p-2}\}$.
 Since $G$ is $(p-1)$-regular, for each $u_i\in V(M_1)$, we have $N_G[u_i]=V(M_1)\cup \{u, v_1\}$. Thus, $|N_G(u)\cap N_G(v_2)|\le |N_G(u)\backslash V(M_1)|=1$, which means $G[N_G(u)\cap N_G(v_2)]$ contains no $K_{p-2}$-copy, a contradiction to Claim~\ref{clm:KP+}.

Thus, if $G$ is connected, then $e(G)> h(n,p)$.

\noindent\textbf{Case 2.} $G$ is disconnected.

Suppose that $G_{1}, G_{2},\ldots, G_{s}\,(s\geq 2)$ are all components in $G$.
Assume that there exist $1\le i<j\le s$ such that $G_i$ and $G_j$ are both $K_p$-free.
Let $u\in V(G_i)$ and $v\in V(G_j)$. Then $G+uv$ has a $K_p^{+1}$-copy, say $M$, containing $uv$.
Since $uv$ is not contained in any triangle of $G+uv$, $uv$ must be a pendant edge of $M$.
We may assume that $d_M(u)=p$ and $d_M(v)=1$. Then $u$ is a vertex of the $K_p$-copy in $M$, which implies that $G_i$ contains a $K_p$-copy, a contradiction. So, $G$ has at most one $K_p$-free component.

Also, each component containing $K_p$ has exactly $p$ vertices, for otherwise $G$ contains a $K_p^{+1}$-copy, a contradiction.
So, we may assume that $G_i\cong K_p$ for each $i\in [s-1]$.
For the remaining component $G_s$ of order $n_s$, if $n_s\le p$, then $G_s\cong K_{n_s}$, for otherwise adding a non-edge of $G_s$ produces no $K_p^{+1}$-copy, a contradiction. Thus, $e(G)=h(n,p)$.
If $n_s\ge p+1$, then $G_s$ is $K_p^{+1}$-saturated and $n_s\equiv k\pmod{p}$. Then by the conclusion of Case 1, we have $e(G_s)>h(n_s,p)$, and thus,
\[
\begin{aligned}
e(G) &\ge \frac{(p-1)(n-n_s)}{2}+e(G_s) \\
& >\frac{(p-1)(n-n_s)}{2}+h(n_s,p)\\
& =\frac{(p-1)(n-n_s)}{2}+\frac{(p-1)n_s}{2}+\frac{k(k-p)}{2} \\
& = h(n,p).
\end{aligned}
\]
Therefore, if $p\ge4$, then $e(G)\ge h(n,p)$, and the equality holds only when
$G\cong\lfloor\frac{n}{p}\rfloor K_p\cup K_k$.

So, $\sat(n, K_p^{+1})=h(n,p)$ and $\SAT(n, K_p^{+1})=\big\{\lfloor\frac{n}{p}\rfloor K_p\cup K_k\big\}$, as required.
\end{proof}

For $p=3$, the graph $K_p^{+1}$ is $J_{1,1}$, and its extremal graphs are given by Theorem~\ref{thm:matching-join}. Also, when $1\le n\le p$, the only $K_p^{+1}$-saturated graph is $K_n$, so the formula in Proposition~\ref{prop:one-leaf} extends to these orders as well.

\section{Saturation number for $J_{p, q}$}\label{sec:matchings}

We prove Theorem~\ref{thm:matching-join}. We first consider the case $p=1$. For $p\ge2$, we show that each $J_{p,q}$-saturated graph $G$ with $e(G)\le n+3p-4$ has a full-degree vertex and reduce the problem to the saturation number of a matching. Finally, we prove that the condition on $n$ in Theorem~\ref{thm:matching-join} is sharp.
\subsection{The case \texorpdfstring{$p=1$}{p=1}}
\begin{proof}[\textbf{Proof of Theorem~\ref{thm:matching-join} for $p=1$}]
Since $n\ge q+3$, the star $S_n$ is a $J_{1,q}$-saturated graph with $n-1$ edges. So, $\sat(n, J_{1,q})\le n-1$. Now we prove the lower bound. Let $G$ be a minimum $J_{1,q}$-saturated graph of order $n$. Then $e(G)=\sat(n,J_{1,q})\le n-1$.

If $G$ is connected, then $e(G)= n-1$ and $G$ is a tree. Thus, for any $uv\in E(\overline{G})$, $uv$ belongs to the triangle of the $J_{1,q}$-copy in $G+uv$, which implies $\diam(G)=2$. Since $G$ is a tree and $\diam(G)=2$, the longest path in $G$ has length 2. Thus, $G\cong S_n$.

Now, suppose that $G$ is disconnected and $G_{1}, G_{2},\ldots, G_{s}\,(s\geq 2)$ are all components in $G$. If $\delta(G)\ge 2$, then $e(G)\ge n$, a contradiction. So, we may assume that $u$ is a vertex in $G_1$ with $d_{G}(u)\le 1$.
Take $v\in V(G)\backslash V(G_1)$. Then $G+uv$ has a $J_{1,q}$-copy, say $H$, containing $uv$. Since $uv$ is not contained in any triangle of $G+uv$, $uv$ must be a pendant edge in $H$. As $d_{G+uv}(u)\le 2$, we have $d_H(u)=1$ and $d_H(v)=q+2$. Then $d_G(v)\ge q+1$ for each $v\in V(G)\backslash V(G_1)$.
Let $|V(G_1)|=n_1$. Since $G_1$ is connected, we have $e(G_1)\ge n_1-1$, and thus,
\[
e(G)\ge e(G_1)+\frac{1}{2}\sum\limits_{v\in V(G)\backslash V(G_1)}d_G(v)\ge (n_1-1)+\frac{(q+1)(n-n_1)}{2}
=n-1+\frac{(q-1)(n-n_1)}{2}\ge n-1.
\]
Suppose that $e(G)=n-1$. Then $e(G_1)= n_1-1$, $q=1$ and $d_G(v)= 2$ for each $v\in V(G)\backslash V(G_1)$. Recall that $uv$ is a pendant edge of $H$. Since $q=1$, $G[N_G[v]]$ contains a $K_3$-copy.
Since $d_G(v)=2$, we have $G[N_G[v]]\cong K_3$, and thus $G_2\cup\cdots\cup G_s\cong(s-1)K_3$.
For the remaining component $G_1$, if $n_1\le 3$, then $G_1$ must be complete since $G$ is $J_{1,q}$-saturated. By our assumption that $e(G_1)= n_1-1$, we have $G_1\cong K_1$ or $K_2$. If $n_1\ge 4$, $G_1$ must be a connected $J_{1,q}$-saturated graph. By the argument above, $G_1\cong S_{n_1}$.

Therefore, $\sat(n,J_{1,q})=n-1$. Moreover, if $q\ge2$, $S_n$ is the unique minimum $J_{1,q}$-saturated graph of order $n$. If $q=1$, it can be checked that each graph $aK_3\cup S_{n-3a}$ with $a\ge0$, $n>3a$ and $n-3a\ne3$ is $J_{1,1}$-saturated. So, $\SAT(n,J_{1,1})=\{aK_3\cup S_{n-3a}:a\ge0,\ n>3a,\ n-3a\ne3\}.$
\end{proof}

\subsection{Some lemmas for \texorpdfstring{$p\ge2$}{p at least 2}}
\begin{lemma}\label{lem:Tutte-Berge} \text{\bf \cite{Berge1958}}
Let $G$ be a graph of order $n$. Then
\[
\nu(G)=\frac{1}{2}\left(n-\max_{U\subseteq V(G)}
\bigl(\odd(G-U)-|U|\bigr)\right).
\]
In particular,  there exists a set $U\subseteq V(G)$ such that
$\odd(G-U)-|U|=n-2\nu(G)$.
\end{lemma}

\begin{lemma}\label{lem:matching}
Let $r\ge2$, $n\ge\max\{2r,5\}$, and let $G$ be an $rK_2$-saturated
graph of order $n$. Then $e(G)\ge3r-3$. The equality holds if and only if $n\ge3r-3$ and
$G\cong(r-1)K_3\cup\overline K_{n-3r+3}$. In particular, if $2r\le n<3r-3$, then $e(G)>3r-3$.
\end{lemma}

\begin{proof}
Since $n\ge2r$ and $G$ is $rK_2$-free, $G$ is not complete.
For any $e\in E(\overline G)$, since $G+e$ contains an $rK_2$-copy, and adding $e$ to $G$ increases the matching number by at most one, we have $\nu(G)=r-1$.
By Lemma~\ref{lem:Tutte-Berge}, there exists a set $U\subseteq V(G)$ such that
$\odd(G-U)-|U|=n-2r+2$. Let $t=|U|$ and $h=\odd(G-U)$. Then $h-t=n-2r+2$.
Since $n\ge2r$, $G-U$ has at least $t+2$ odd components.

For any $xy\in E(\overline G)$, since $G$ is $rK_2$-saturated,
$\nu(G+xy)=r$. By Lemma~\ref{lem:Tutte-Berge}, there exists
$U_{xy}\subseteq V(G)$ such that
$\odd((G+xy)-U_{xy})-|U_{xy}|=n-2r$. Therefore,
\begin{equation}\label{eq:matching-tutte}
\odd((G+xy)-U)-|U|\le \odd((G+xy)-U_{xy})-|U_{xy}|=n-2r.
\end{equation}

We first show that each vertex of $U$ is a full-degree vertex of $G$. Otherwise, we can assume that $x\in U$ and $xy\in E(\overline G)$. Since $(G+xy)-U=G-U$, $\odd((G+xy)-U)-|U|=n-2r+2$, a contradiction to \eqref{eq:matching-tutte}. Thus, each vertex of $U$ is full-degree.

We next show that each component of $G-U$ is complete. Suppose that $x$ and $y$ are two non-adjacent vertices in the same component of $G-U$.
Then $\odd((G+xy)-U)-|U|=\odd(G-U)-|U|=n-2r+2$, a contradiction to \eqref{eq:matching-tutte}. Thus, each component of $G-U$ is complete.

Also, $G-U$ has no component with even number of vertices. Otherwise,
let $Q$ be an even component of $G-U$. Recall that $G-U$ has an odd component. Let $Q'$ be such an odd component. Take $x\in V(Q)$ and $y\in V(Q')$. Then $xy\in E(\overline G)$. Since $Q$ and $Q'$ are replaced by one odd component in $(G+xy)-U$,
$\odd((G+xy)-U)=\odd(G-U)$, and hence $\odd((G+xy)-U)-|U|=n-2r+2$, a contradiction to \eqref{eq:matching-tutte}.
Consequently, all components of $G-U$ are odd complete graphs. It follows that
$G\cong K_t\vee\Big(\bigcup_{i=1}^h K_{2a_i+1}\Big)$, where $a_i\ge0$.
Since $n=t+2\sum_{i=1}^h a_i+h$ and $h-t=n-2r+2$, we have
$\sum_{i=1}^h a_i=r-1-t$. Thus, $0\le t\le r-1$.
Then
\begin{equation}\label{eq:tutte}
\begin{aligned}
e(G)&=\binom{t}{2}+t(n-t)+\sum_{i=1}^h\binom{2a_i+1}{2}\\
&=tn-\binom{t+1}{2}+\sum_{i=1}^h a_i(2a_i+1)\\
&\ge tn-\binom{t+1}{2}+3(r-1-t)\\
&=3r-3+t\left(n-3-\frac{t+1}{2}\right).
\end{aligned}
\end{equation}

Since $t\le r-1$ and $n\ge\max\{2r,5\}$, we have
\[n-3-\frac{t+1}{2}\ge \max\{2r,5\}-3-\frac r2>0.\]
Therefore, $e(G)\ge3r-3$.
Suppose that the equality holds. Then the equality in \eqref{eq:tutte} holds throughout, and hence $t=0$ and $a_i\in\{0,1\}$ for each $i$. Since
$\sum_{i=1}^h a_i=r-1$, exactly $r-1$ components of $G$ are
triangles and all remaining components are isolated vertices. Thus,
$G\cong(r-1)K_3\cup\overline K_{n-3r+3},$ and necessarily $n\ge3r-3$.

Conversely, let $G\cong(r-1)K_3\cup\overline K_{n-3r+3}$. Clearly, $\nu(G)=r-1$. For any $xy\in E(\overline G)$, the vertices $x$ and $y$ belong to different components. From each of the $r-1$ triangles, we can choose an edge disjoint from $\{x,y\}$. Together with $xy$, these edges form a matching of size $r$ in $G+xy$.  Thus, $G$ is $rK_2$-saturated with $3r-3$ edges.

In particular, if $2r\le n<3r-3$ and $e(G)=3r-3$, then equality holds in \eqref{eq:tutte}. By the equality case above, this implies $n\ge3r-3$, a contradiction. Therefore, $e(G)>3r-3$.
\end{proof}
The following lemma plays an important role in our proof of Theorem~\ref{thm:matching-join}.

\begin{lemma}\label{lem:full-degree}
Let $p\ge2$, $q\ge1$, and $n\ge2p+q+1$. If $G$ is a $J_{p,q}$-saturated graph of order $n$ with $e(G)\le n+3p-4$, then $G$ has a full-degree vertex.
\end{lemma}

\begin{proof}
Let $D=2p+q$. Let the vertex of degree $D$ in $J_{p,q}$ be its center. Suppose that $G$ has no full-degree vertex. Then $G$ is not complete.

First assume that $n=D+1$. For any $xy\in E(\overline G)$, $G+xy$ contains a $J_{p,q}$-copy. Since $|V(J_{p,q})|=D+1=n$, $G+xy$ has a vertex of degree $n-1$. Since $G$ has no full-degree vertex, this vertex is $x$ or $y$. Thus, $d_G(x)=n-2$ or $d_G(y)=n-2$. So, each non-edge of $G$ contains an $(n-2)$-degree vertex.
Since each $(n-2)$-degree vertex belongs to exactly one non-edge, the number of non-edges of $G$ is at most the number of $(n-2)$-degree vertices in $G$. So, $e(\overline{G})\le n$.
However, since $n=D+1=2p+q+1\ge2p+2$ and $p\ge2$, we have
\[
\binom{n}{2}-n-3p+4-n
=\frac{n(n-5)}{2}-3p+4
\ge\frac{(2p+2)(2p-3)}{2}-3p+4
=2(p-1)^2-1>0.
\]
Thus,
\[
e(\overline G)=\binom{n}{2}-e(G)\ge\binom{n}{2}-n-3p+4>n,
\]
a contradiction.  So, if $n=D+1$, $G$ has a full-degree vertex.

Now, assume that $n\ge D+2$.
Since $G$ is $J_{p,q}$-free, if $x$ is a vertex in $G$ such that $G[N_G(x)]$ contains a $pK_2$-copy, then $d_G(x)\le D-1$. Define
\[
A=\{x\in V(G):d_G(x)=D-1\text{ and }\nu(G[N_G(x)])\ge p\},\qquad B=V(G)\setminus A.
\]
\begin{claim}\label{clm:Jpqcenter}
Let $xy\in E(\overline G)$, and let $H$ be a $J_{p,q}$-copy in $G+xy$ containing $xy$. Suppose that $xy$ is a pendant edge of $H$, and let $z\in\{x,y\}$ be the center of $H$. Then $z\in A$.
\end{claim}

\begin{proof}
Since $xy$ is a pendant edge of $H$, $\nu(G[N_G(z)])\ge p$. Since $G$ is $J_{p,q}$-free, we have $d_G(z)\le D-1$. As $d_H(z)=D$, we have $d_G(z)\ge D-1$. Thus $d_G(z)=D-1$, and hence $z\in A$.
\end{proof}

\begin{claim}\label{clm:matching-connected}
$G$ is connected and $|A|\le2$.
\end{claim}
\begin{proof}
If $B=\emptyset$, then
\[
e(G)=\frac{(D-1)n}{2}>n+3p-4,
\]
a contradiction. Suppose that $G$ is disconnected. Take $x\in B$, and let $G^*$ be the component containing $x$. For any $y\notin V(G^*)$, the edge $xy$ belongs to no triangle of $G+xy$, so it must be a pendant edge of a new $J_{p,q}$-copy. By Claim~\ref{clm:Jpqcenter}, $x\in A$ or $y\in A$. Since $x\in B$, we have $y\in A$, and thus, $d_G(y)=D-1$.

Let $s=|V(G)\backslash V(G^*)|$. Since each vertex in $V(G)\backslash V(G^*)$ has degree $D-1$, we have $s\ge D$. Also, $G^*$ is connected. Since $D\ge2p+1$ and $\frac{D(D-3)}2-3p+3\ge2(p-1)^2>0$, we have
$$e(G)\ge (n-s-1)+\frac{(D-1)s}{2}=n-1+\frac{(D-3)s}{2}\ge n-1+\frac{D(D-3)}2>n+3p-4,$$
a contradiction. So, $G$ is connected.

If $A=\emptyset$, there is nothing further to prove. Otherwise, let $G'$ be the graph obtained from $G$ by successively deleting vertices of degree one until no leaf exists. For each $a\in A$, let $M$ be a $pK_2$-copy in $G[N_G(a)]$. Since each edge of $M$ belongs to a triangle containing $a$, none of these vertices is deleted. So, $d_{G'}(a)\ge 2p$.
Moreover, $G'$ is not empty. Since $G$ is connected, $G'$ is connected. Since $G'$ has no leaf, $\delta(G')\ge2$. Each deletion removes one vertex and one edge, so
\[
e(G')-|V(G')|=e(G)-n\le3p-4.
\]
Consequently,
\[
|A|(2p-2)\le\sum_{x\in V(G')}(d_{G'}(x)-2)
\le6p-8.
\]
Since $p\ge2$, this gives $|A|\le2$.
\end{proof}

Let
\[
C=\{x\in V(G):d_G(x)\ge D-1\}.
\]
Since $G$ is not complete, we may choose $xy\in E(\overline G)$. The center of a new $J_{p,q}$-copy in
$G+xy$ has degree at least $D-1$ in $G$. Hence $C\ne\emptyset$.
Suppose that $C=\{u\}$. For any $x\ne u$, since $d_G(x)\le D-2<n-2$, there exists
$y\in V(G)\setminus\{u,x\}$ such that $xy\notin E(G)$. Since $C=\{u\}$, the center of any new $J_{p,q}$-copy in $G+xy$ must be $u$. Hence $ux\in E(G)$. As $x$ was arbitrary, $u$ is a full-degree vertex of $G$, a contradiction. Therefore, $|C|\ge2$.
\begin{claim}\label{clm:matching-mindegree}
$\delta(G)=2$.
\end{claim}
\begin{proof}
By Claim~\ref{clm:matching-connected}, $\delta(G)\ge1$. If $\delta(G)\ge3$, since $|C|\ge 2$, we have
\[
2e(G)\ge3(n-2)+2(D-1)=3n+4p+2q-8.
\]
Since $e(G)\le n+3p-4$, we have $n\le2p-2q$, contrary to $n\ge D+2=2p+q+2$. Hence $\delta(G)\le2$.

Suppose that $x$ is a leaf in $G$, and let $v$ be its unique neighbour. Let $Y=V(G)\setminus N_G[v]$. Since $G$ has no full-degree vertex,  $Y\neq\emptyset$.
For each $y\in Y$, $xy$ belongs to no triangle of $G+xy$. Thus, $xy$ must be a pendant edge of a new $J_{p,q}$-copy in $G+xy$. Since $d_{G+xy}(x)=2<D$, $y$ is the center of $J_{p,q}$. By Claim~\ref{clm:Jpqcenter}, $y\in A$.
Since $y$ was arbitrary in $Y$, $Y\subseteq A$. Moreover, by Claim~\ref{clm:matching-connected}, $|Y|\le|A|\le 2$.

If $|Y|=1$, let $Y=\{a\}$. Then $d_G(v)=n-2\ge D$. Since $a\in A$, let $M$ be a matching of size $p$ in $G[N_G(a)]$. Since $M\subseteq G[N_G(v)]$ and $d_G(v)\ge D$, $v$ is the center of a $J_{p,q}$-copy in $G$, a contradiction.

Thus, $|Y|=2$. Let $Y=\{a_1,a_2\}$ and take a matching $M$ of size $p$ in $G[N_G(a_1)]$. Since $v$ is adjacent to each vertex of $G$ except $a_1$ and $a_2$, at least $2p-1$ vertices of $V(M)$ are adjacent to $v$. Thus, $G[\{v,a_1\}\cup V(M)]$ is connected, has $2p+2$ vertices, and
$$e(G[\{v,a_1\}\cup V(M)])\ge e(\{v\},V(M))+e(\{a_1\},V(M))+e(M)\ge(2p-1)+2p+p=5p-1.$$
Since $G$ is connected, $e(G)\ge(5p-1)+n-(2p+2)=n+3p-3,$ a contradiction.

So, $\delta(G)=2$.
\end{proof}

\begin{claim}\label{clm:matching-centers}
$|C|=2$, and both vertices in $C$ have degree at least $D$.
\end{claim}
\begin{proof}
If $|C|\ge3$, then by Claim~\ref{clm:matching-mindegree},
\[
2e(G)\ge2(n-3)+3(D-1)=2n+3(2p+q)-9>2n+6p-8,
\]
a contradiction. Since $|C|\ge2$, we have $|C|=2$. Let $C=\{u,v\}$ and $G_1=G-\{u,v\}$.
Each vertex of $G_1$ has degree at most $D-2$, while $|V(G_1)|=n-2\ge D$. Thus, $G_1$ is not complete.

Let $xy\in E(\overline{G_1})$. Since neither $x$ nor $y$ can be
the center of a new $J_{p,q}$-copy in $G+xy$, its center is $u$ or
$v$. Thus, at least one of $u,v$ already has degree at least $D$ in $G$. We may assume that $d_G(u)\ge D$.

Suppose that $d_G(v)=D-1$. Then each non-edge of $G_1$ must produce a new
$J_{p,q}$-copy with center $u$. If some $x\in V(G_1)$ is not
adjacent to $u$, then $x$ is adjacent to each other vertex of
$G_1$. Hence
$d_G(x)\ge n-3\ge D-1$, a contradiction since $x\notin C$. Thus, $u$ is
adjacent to each vertex of $G_1$. Since $G$ has no full-degree
vertex, $uv\notin E(G)$. Therefore,
\[
e(G_1)= e(G)-d_G(u)-d_G(v)\le n+3p-4-(n-2)-(D-1)=p-q-1.
\]
For any $xy\in E(\overline{G_1})$, we have
$e(G_1+xy)=e(G_1)+1\le p-q<p$.
Hence $G_1+xy$ contains no matching of size $p$, and thus $G+xy$ contains no $J_{p,q}$-copy, a contradiction. Thus $d_G(v)\ge D$ and the claim holds.
\end{proof}
Let $C=\{u,v\}$ and $G_1=G-\{u,v\}$. Each vertex of $G_1$ is adjacent to at least one of $u,v$. Otherwise, let $w\in V(G_1)$ be adjacent to neither $u$ nor $v$. Since $d_G(w)\le D-2<n-3$, there exists $w'\in V(G_1)\setminus\{w\}$ such that $ww'\notin E(G)$. Then neither $u$ nor $v$ can be the center of a new $J_{p,q}$-copy in $G+ww'$, a contradiction.

Let $X$ consist of the vertices of $G_1$ adjacent only to $u$, let $Y$ consist of those adjacent only to $v$, and let $Z$ consist of those adjacent to both. Write $x=|X|, y=|Y|, z=|Z|$. Let $\varepsilon=1$ if $uv\in E(G)$ and $\varepsilon=0$ otherwise.

Since $e(G)=e(G_1)+x+y+2z+\varepsilon=e(G_1)+n-2+z+\varepsilon\le n+3p-4$,
we have
\begin{equation}\label{eq:matching-G1-budget}
e(G_1)+z+\varepsilon\le3p-2.
\end{equation}
Also, since $d_G(u)\ge D$ and $d_G(v)\ge D$ by Claim~\ref{clm:matching-centers},
\begin{equation}\label{eq:matching-center-degrees}
x+z+\varepsilon\ge D,\qquad y+z+\varepsilon\ge D.
\end{equation}

\begin{claim}\label{clm:matching-low-in-Z}
Each vertex of degree two in $G$ belongs to $Z$.
\end{claim}
\begin{proof}
Suppose that $w\in X$ has degree two. Then $d_{G_1}(w)=1$. We claim that every vertex of $X$ is adjacent to every vertex of $Y$. Otherwise, let $x_1\in X$ and $y_1\in Y$ such that $x_1y_1\notin E(G)$. Then neither $u$ nor $v$ can be the center of a new $J_{p,q}$-copy in $G+x_1y_1$, a contradiction. Then $y\le1$. By \eqref{eq:matching-center-degrees},
$z+\varepsilon\ge D-1$. Together with \eqref{eq:matching-G1-budget}, we have
\[
e(G_1)\le p-q-1.
\]

Since $wv\in E(\overline G)$, the center of a new $J_{p,q}$-copy in $G+wv$ must be $u$ or $v$. In either case, at least $p-\varepsilon$ edges of the corresponding matching belong to $G_1$, which implies $e(G_1)+\varepsilon\ge p$, contrary to $e(G_1)\le p-q-1$. By symmetry, no vertex of $Y$ has degree two. Together with $d_G(u)\ge D$ and $d_G(v)\ge D$, the claim holds.
\end{proof}
By Claims~\ref{clm:matching-mindegree} and~\ref{clm:matching-low-in-Z},
there is a vertex $w\in Z$ with degree two in $G$. Thus, $w$ is
isolated in $G_1$.

Let
\[
G_u=G_1[X\cup Z],\qquad
G_v=G_1[Y\cup Z],\qquad
k=p-1-\varepsilon.
\]
Since $d_G(u),d_G(v)\ge D$ and $G$ is $J_{p,q}$-free, neither
$G[N_G(u)]$ nor $G[N_G(v)]$ contains a matching of size $p$.
If $\varepsilon=0$, then
\[
\nu(G_u),\nu(G_v)\le p-1=k.
\]
Suppose that $\varepsilon=1$. If $\nu(G_u)\ge p-1$, let $M$ be a
matching of size $p-1$ in $G_u$. Since $w$ is isolated in $G_1$,
$w\notin V(M)$, and hence $M\cup\{vw\}$ is a matching of size $p$
in $G[N_G(u)]$, a contradiction. Thus, $\nu(G_u)\le p-2=k$.
Similarly, $\nu(G_v)\le p-2=k$. Therefore,
\begin{equation}\label{eq:matching-local-matching}
\nu(G_u),\nu(G_v)\le k.
\end{equation}
\begin{claim}\label{clm:matching-no-G1-leaf}
$G_1$ has no vertex of degree one.
\end{claim}

\begin{proof}
Suppose that $d_{G_1}(a)=1$, and let $b$ be the unique neighbour of $a$ in $G_1$. By
Claim~\ref{clm:matching-low-in-Z}, each vertex of $X\cup Y$ has degree at least three in $G$, and hence has degree at least two in $G_1$. Thus, $a\in Z$.

Consider the non-edge $wb$. If $b\in X$, then the center of a new $J_{p,q}$-copy in $G+wb$ must be $u$, and $wb$ belongs to its matching of size $p$.
After deleting $wb$, at least $p-1-\varepsilon=k$ edges of this matching form a matching in
$G_u-b$. Since $a$ is an isolated vertex of $G_1-b$, none of these $k$ edges is incident with $a$. Together with $ab$, they form a matching of size $k+1$ in $G_u$, contradicting \eqref{eq:matching-local-matching}.

Similarly, if $b\in Y$, then $G_v$ contains a matching of size $k+1$, a contradiction. Suppose that $b\in Z$. The center of a new $J_{p,q}$-copy in $G+wb$ is $u$ or $v$. As above, $G_u$ or $G_v$
contains a matching of size $k+1$, again contradicting \eqref{eq:matching-local-matching}.

Therefore, $G_1$ has no vertex of degree one.
\end{proof}

By Claim~\ref{clm:matching-low-in-Z}, each vertex of $X\cup Y$ has
degree at least two in $G_1$. Hence $2e(G_1)\ge 2x+2y$ and
\[
e(G_1)\ge x+y=n-2-z.
\]
Together with $e(G_1)+z+\varepsilon\le3p-2$ by \eqref{eq:matching-G1-budget}, this gives
\begin{equation}\label{eq:matching-order-upper}
n\le3p-\varepsilon.
\end{equation}

Take any $t\in V(G_1)\setminus\{w\}$. Since $wt\in E(\overline G)$,
a new $J_{p,q}$-copy in $G+wt$ has center $u$ or $v$, and $wt$
belongs to its matching of size $p$. After deleting $wt$, at least
$p-1-\varepsilon=k$ edges of this matching lie in $G_1$. Thus,
$G_1$ contains a matching $M$ of size $k$. By
Claim~\ref{clm:matching-no-G1-leaf}, each vertex of $V(M)$ has degree
at least two in $G_1$. Hence
$2e(G_1)\ge\sum_{x\in V(M)}d_{G_1}(x)\ge4k$, and thus
\begin{equation}\label{eq:matching-Egde-lower}
e(G_1)\ge2k=2p-2-2\varepsilon.
\end{equation}
Adding the two inequalities in
\eqref{eq:matching-center-degrees} and using
$x+y=n-2-z$, we have
\begin{equation}\label{eq:matching-z-lower}
z\ge2D-n+2-2\varepsilon.
\end{equation}
By \eqref{eq:matching-Egde-lower}, $e(G_1)\ge2p-2-2\varepsilon$. Together with $e(G_1)\le3p-2 -z-\varepsilon$ by \eqref{eq:matching-G1-budget}, we have $z\le p+\varepsilon$.
Thus, by \eqref{eq:matching-z-lower},
\[
n\ge3p+2q+2-3\varepsilon=3p+2-\varepsilon+(2q-2\varepsilon)\ge3p+2-\varepsilon.
\]
This contradicts
\eqref{eq:matching-order-upper}. Therefore, $G$ has a full-degree
vertex.
\end{proof}

\subsection{Proof for \texorpdfstring{$p\ge2$}{p at least 2}}
\begin{proof}[\textbf{Proof of Theorem~\ref{thm:matching-join} for $p\ge2$}]
Suppose first that $n\ge\max\{2p+q+1,3p-2\}$. Let
\[
G_0=K_1\vee\bigl((p-1)K_3\cup\overline K_{n-3p+2}\bigr).
\]
Let $v$ be the full-degree vertex of $G_0$. Since $\nu(G_0[N_{G_0}(v)])=p-1$, and each vertex in $V(G_0)\backslash\{v\}$ has degree at most three, $G_0$ is $J_{p,q}$-free. For any $xy\in E(\overline G_0)$, $G_0+xy-v$ contains a matching of size $p$. Since $n-1\ge2p+q$, $G_0+xy$ contains a $J_{p,q}$-copy with center $v$. Thus, $G_0$ is $J_{p,q}$-saturated and
\[
\sat(n,J_{p,q})\le e(G_0)=n+3p-4.
\]

Let $G$ be a minimum $J_{p,q}$-saturated graph of order $n$. Then $e(G)\le n+3p-4$. By Lemma~\ref{lem:full-degree}, $G$ has a full-degree vertex $v$. Let $G'=G-v$. Since $n-1\ge2p+q$ and $G$ is $J_{p,q}$-free, $G'$ contains no matching of size $p$.

Let $e\in E(\overline {G'})$. Since $G$ is $J_{p,q}$-free, $G+e$ contains a $J_{p,q}$-copy $H$ containing $e$. Let $z$ be the center of $H$, and let $M$ be the matching of size $p$ in $H-z$. If $v\notin V(M)$, then $M$ is contained in $G'+e$. If $vw\in M$, then replacing $vw$ by $zw$ gives a matching of size $p$ in $G'+e$. Hence $G'$ is $pK_2$-saturated.

Since $n-1\ge2p+q\ge\max\{2p,5\}$, by Lemma~\ref{lem:matching},
\[
e(G)=n-1+e(G')\ge n-1+3p-3=n+3p-4.
\]
Therefore, $\sat(n,J_{p,q})=n+3p-4$. Moreover, the equality in Lemma~\ref{lem:matching} gives
$G'\cong(p-1)K_3\cup\overline K_{n-3p+2}$,
and hence
$G\cong K_1\vee\bigl((p-1)K_3\cup\overline K_{n-3p+2}\bigr).$

Suppose that $2p+q+1\le n<3p-2$. If $\sat(n,J_{p,q})\le n+3p-4$, let $G$ be a minimum $J_{p,q}$-saturated graph of order $n$. By Lemma~\ref{lem:full-degree}, $G$ has a full-degree vertex $v$. As above, $G-v$ is $pK_2$-saturated. Since $2p\le n-1<3p-3$, Lemma~\ref{lem:matching} gives $e(G-v)>3p-3$. Thus,
\[
e(G)=n-1+e(G-v)>n+3p-4,
\]
a contradiction. Therefore, $\sat(n,J_{p,q})>n+3p-4$. This completes the proof.
\end{proof}

\subsection{Sharpness of the order condition}\label{subsec:sharpness}
Theorem~\ref{thm:matching-join} determines exactly when $\sat(n,J_{p,q})=n+3p-4$ holds for $n\ge2p+q+1$. If $1\le q\le p-4$ and $2p+q+1\le n\le3p-3$, then by Theorem~\ref{thm:matching-join},
\[
\sat(n,J_{p,q})>n+3p-4.
\]
Thus, the bound $n\ge3p-2$ cannot be lowered. If $q\ge p-3$, then $n\ge2p+q+1$ implies $n\ge3p-2$.

The bound $n\ge2p+q+1$ also cannot be lowered. If $n=2p+q$, then
$|V(J_{p,q})|=n+1$. Thus, adding an edge to a non-complete graph of
order $n$ cannot give a $J_{p,q}$-copy, and hence $K_n$ is the only
$J_{p,q}$-saturated graph of order $n$. Moreover, since $n=2p+q$, for fixed $p$ the left-hand side below is increasing in $q$. Hence its minimum is attained when $q=1$, and
\[
\binom n2-(n+3p-4)\ge2(p-1)^2+1>0.
\]
Therefore, the formula does not hold for $n=2p+q$, and the
order condition in Theorem~\ref{thm:matching-join} is sharp.

For $(p,q)=(1,1)$, the extremal graph is not unique, since both
$K_3\cup K_1$ and $S_4$ are extremal graphs for $J_{1,1}$.
\section{Concluding remarks}\label{sec:conclusion}
We conclude with several observations and open problems.

Theorem~\ref{thm:Kp+} determines the saturation number for $K_p^{+q}$ when $q\ge p-1$ and $n\ge\max\{p+q,\binom p2\}$. Proposition~\ref{prop:one-leaf} gives the result for $q=1$. The following ranges are not settled here.

\begin{problem}\label{prob:clique-remaining}
Determine $\sat(n,K_p^{+q})$ and characterize the extremal graphs for $p\ge4$ in the following ranges:
\begin{enumerate}[label=(\roman*)]
\item $2\le q\le p-2$ and $n\ge p+q$;
\item $q\ge p-1$ and $p+q\le n<\binom p2$.
\end{enumerate}
\end{problem}

The condition $q\ge p-1$ cannot be reduced to $q\ge p-2$ if both the formula and the uniqueness statement in Theorem~\ref{thm:Kp+} are kept. To see this, let $s=p-2$ and $q=s$. For every integer $a\ge1$, let
\[
G=aK_{2s+1}\cup K_s.
\]
Obviously, $G$ is $K_{s+2}^{+s}$-saturated. Moreover, the graph has $n=a(2s+1)+s$ vertices and
\[
a\binom{2s+1}{2}+\binom{s}{2}
=sn-\binom{s+1}{2}
\]
edges. The graph is disconnected, whereas the extremal graph in Theorem~\ref{thm:Kp+} is connected. Since these examples have arbitrarily large order and the same number of edges as in Theorem~\ref{thm:Kp+}, the formula and the uniqueness statement in that theorem cannot both be extended to $q=p-2$. Also, this construction alone does not determine $\sat(n,K_p^{+(p-2)})$.

The natural condition $n\ge p+q$ in Theorem~\ref{thm:Kp+} is sharp for the stated formula. If $n=p+q-1$, then the only $K_p^{+q}$-saturated graph is $K_n$, and
\[
\binom n2-\left((p-2)n-\binom{p-1}{2}\right)
=\binom{q+1}{2}>0.
\]
Thus, the formula in Theorem~\ref{thm:Kp+} does not hold when $n=p+q-1$, and hence the bound $n\ge p+q$ cannot be lowered. In contrast, the condition $n\ge\binom p2$ is only sufficient. Whether the condition $n\ge\binom p2$ can be improved remains open; see Problem~\ref{prob:clique-remaining}(ii).

For the matching family, Theorem~\ref{thm:matching-join} determines exactly when the formula $\sat(n,J_{p,q})=n+3p-4$ holds in the natural range $n\ge2p+q+1$. If $1\le q\le p-4$ and $2p+q+1\le n\le3p-3$, the theorem gives only the strict inequality $\sat(n,J_{p,q})>n+3p-4$, while the exact values remain unknown. Thus the remaining question is as follows.

\begin{problem}\label{prob:matching-remaining}
Determine $\sat(n,J_{p,q})$ and characterize the extremal graphs for
\[
p\ge5,\qquad1\le q\le p-4,\qquad
2p+q+1\le n\le3p-3.
\]
\end{problem}

We next consider the upper bound of Cameron and Puleo~\cite{Ca}. If $F=F'\cup qK_1$ and $n\ge |V(F)|$, then an $n$-vertex graph contains $F$ if and only if it contains $F'$. The same equivalence holds after adding an edge. Therefore
\[
\sat(n,F'\cup qK_1)=\sat(n,F').
\]
For $F=pK_2\cup qK_1$, by Theorem~\ref{thm:matching-join} and Lemma~\ref{lem:matching}, together with the case $p=1$, we have
\[
\sat(n,K_1\vee F)=n-1+\sat(n-1,F)
\]
under the order condition of the theorem.

For $F=K_{p-1}\cup qK_1$ under the assumptions of Theorem~\ref{thm:Kp+}, by the equality above and the known value of $\sat(n-1,K_{p-1})$ in \cite{erdos1964}, we have
\[
n-1+\sat(n-1,F)=n-1+(p-3)(n-1)-\binom{p-2}{2}=(p-2)n-\binom{p-1}{2}=\sat(n,K_1\vee F).
\]
In contrast, for $q=1$ and $n\ge p+1$, the comparison in the proof of Proposition~\ref{prop:one-leaf} shows that the Cameron--Puleo bound is strict.

The following observation applies to an arbitrary graph $F$.

\begin{observation}\label{obs:full-degree-equality}
Let $n>|V(F)|$. If an extremal graph for $K_1\vee F$ has a full-degree vertex, then
\[
\sat(n,K_1\vee F)=n-1+\sat(n-1,F).
\]
\end{observation}
\begin{proof}
Let $G$ be such an extremal graph with a full-degree vertex $v$. Let $G'=G-v$. Then $G'$ is $F$-free. For $e\in E(\overline{G'})$, $G+e$ has a $(K_1\vee F)$-copy, say $H$, with $z$ being the center. If $v$ does not belong to the $F$-copy of $H$, then $G'+e$ contains an $F$-copy. Otherwise, $z\ne v$, and replacing $v$ by $z$ in the $F$-copy of $H$ gives an $F$-copy in $G'+e$. Thus $G'$ is $F$-saturated, and
\[
e(G)=n-1+e(G')\ge n-1+\sat(n-1,F).
\]
The reverse inequality is the Cameron--Puleo bound. This completes our proof.
\end{proof}

Finally, consider the effect of isolated vertices in $F$. Let $F'$ be obtained from $F$ by deleting its isolated vertices. For the matching family, Theorem~\ref{thm:matching-join} and~\cite{Faudree2009b} give
\[
\sat(n,K_1\vee(pK_2\cup qK_1))=\sat(n,K_1\vee pK_2)
\]
for sufficiently large $n$. For $F'=K_{p-1}$ with $p\ge4$, Proposition~\ref{prop:one-leaf} gives
$\sat(n,K_1\vee(F'\cup K_1))<\sat(n,K_1\vee F')$
for $n\ge p+1$. In contrast, Theorem~\ref{thm:Kp+} gives
\[
\sat(n,K_1\vee(K_{p-1}\cup qK_1))=\sat(n,K_1\vee K_{p-1})=\sat(n,K_p)
\]
when $q\ge p-1$ and $n\ge\max\{p+q,\binom p2\}$. This leads to the following problem.

\begin{problem}
For a graph $F'$ and a positive integer $\ell$, under what conditions does
\[
\sat(n,K_1\vee(F'\cup\ell K_1))=\sat(n,K_1\vee F')
\]
hold for all sufficiently large $n$?
\end{problem}

\section*{Acknowledgments}
This research is supported in part by the National Natural Science Foundation of China (No. 12571363) and the Natural Science Foundation of Hunan Province (No. 2025JJ30003).

\section*{Declaration on the use of generative AI}
We did not use any AI tools  for the first draft of this paper. In the current version, the authors used GPT5.5 to  improve the language,  and  the organization of the whole paper. The authors take the full responsibility for the content of the article.

\end{document}